\documentclass{mfat}
\pagespan{169}{183}

\usepackage{mmap}
\usepackage[utf8]{inputenc}

\theoremstyle{plain}
\newtheorem{thm}{Theorem}[section]
\newtheorem*{thm*}{Theorem}
\newtheorem{prop}{Proposition}[section]
\newtheorem*{prop*}{Proposition}
\newtheorem{cor}{Corollary}[section]
\newtheorem*{cor*}{Corollary}

\newtheorem*{lem*}{Lemma}
\theoremstyle{definition}
\newtheorem{defn}{Definition}[section]
\newtheorem*{defn*}{Definition}

\newtheorem*{exmp*}{Example}

\newtheorem*{exmps*}{Examples}

\newtheorem*{rem*}{Remark}

\newtheorem*{rems*}{Remarks}

\newtheorem*{note*}{Note}
\newcommand{\N}{{\mathbb N}}%%%%%%%%%NATURAL NUMBERS
\newcommand{\Z}{{\mathbb Z}}%%%%%%%%%INTEGERS
\newcommand{\R}{{\mathbb R}}%%%%%%%%%REAL NUMBERS
\renewcommand{\theequation}{\thesection.\arabic{equation}}%%%%%%%%%EQUATION NUMBERING WITHIN SECTIONS
\DeclareMathOperator{\Rep}{Re\,} \DeclareMathOperator{\Imp}{Im\,}
\DeclareMathOperator{\dist}{dist} \DeclareMathOperator{\spa}{span}

\begin{document}

\title[On the generation of ultradifferentiable $C_0$-semigroups]
    {On the generation of Beurling type Carleman ultradifferentiable
     $C_0$-semigroups by scalar type spectral operators}

\author{Marat V. Markin}
\address{Department of Mathematics,
California State University, Fresno 5245 N. Backer Avenue, M/S PB
108 Fresno, CA 93740-8001}
%\curraddr{}
\email{mmarkin@csufresno.edu}
%\thanks{Thank you.}

%    General info
\subjclass{Primary 47B40, 47D03, 30D60; Secondary 34G10, 47B15.}
\date{04/06/2015}
%\dedicatory{To Academician Yu. M. Berezansky in honor of his 90th jubilee}
\keywords{Scalar type spectral operator, $C_0$-semigroup
of linear operators, Carleman classes of functions and vectors.}

\begin{abstract}
A characterization of the scalar type spectral generators of
Beurling type Carleman ultradifferentiable $C_0$-semigroups is
established, the important case of the Gevrey
ultradifferentiability is considered in detail, the implementation
of the general criterion corresponding to a certain rapidly
growing defining sequence is observed.
\end{abstract}

\maketitle

%\iffalse

% % % % % % % % % % % % % % % % % % % % % % % % % % % % % % %
\section[Introduction]{Introduction}

The problem of finding conditions necessary and sufficient for a densely defined closed linear operator $A$ in a complex
Banach space $X$ to be the generator of a $C_0$-semigroup $\left\{S(t)
%\middle
|t\ge0\right\}$ with a certain regularity
property such as strong differentiability or analyticity of its orbits on $(0,\infty)$ and thus, of the
\textit{weak/mild solutions} of the associated abstract evolution equation
\begin{equation*}
y'(t)=Ay(t),\quad  t\ge 0,
\end{equation*}
\cite{Ball, Engel-Nagel} is central in qualitative theory.

The well known general generation criteria of analytic and
(infinite) differentiable $C_0$-semigro\-ups
\cite{Hille-Phillips,Pazy1968,Pazy,Yosida1958,Yosida} (cf. also
\cite{Engel-Nagel}) contain restrictions on the location of the
generator's \textit{spectrum} in the complex plane and on its
\textit{resolvent} behavior. As is shown in
\cite{Markin2002(2),Markin2004(1),Markin2008}, when the potential
generators are selected from the class of \textit{scalar type
spectral operators} (see Preliminaries), the restrictions of the
second kind can be dropped in the foregoing and other cases, which
makes the results more transparent, easier to handle, and
inherently qualitative.

The characterization of the scalar type spectral generators of Roumieu type Gevrey ultradifferentiable $C_0$-semigroups found in \cite{Markin2004(1)} is generalized in \cite{Markin2008} to the case of the Roumieu type Carleman ultradifferentiable $C_0$-semigroups. However, neither in  \cite{Markin2004(1)}, nor in \cite{Markin2008}, the case of Beurling type ultradifferentiability has been treated.

In the present paper, we are to establish a generation criterion of a Beurling type Carleman ultradifferentiable $C_0$-semigroup corresponding to a sequence of positive numbers $\left\{m_n \right\}_{n=0}^\infty$ by a scalar type spectral operator, consider in detail the important case of the Gevrey ultradifferentiability, and observe the implementation of the general criterion corresponding to a certain rapidly growing defining sequence.

%%%%%%%%%%%%%%%%%%%%%%%%%%%%%%%%%%%%%%%%%%%%%%%%%%%%%%%%%%%%
\section[Preliminaries]{Preliminaries}

For the reader's convenience, we shall outline here certain  essential preliminaries.

%%%%%%%%%%%%%%%%%%%%%%%%%%%%%%%%%%%%%%%%%%%%%%%%%%%%%%%%%%%%%
\subsection{Scalar Type Spectral Operators}

Henceforth, unless specified otherwise, $A$ is supposed to be a {\it scalar type spectral operator} in a complex
Banach space $(X,\|\cdot\|)$ and $E_A(\cdot)$ to be its {\it spectral measure} (the {\it resolution of the identity}),
the operator's \textit{spectrum} $\sigma(A)$ being the {\it support} for the latter \cite{Survey58,Dun-SchIII}.

In a complex Hilbert space, the scalar type spectral operators are precisely those similar to the {\it normal} ones
\cite{Wermer}.

A scalar type spectral operator in complex Banach space has an {\it operational calculus} analogous to that of a \textit{normal operator} in a complex Hilbert space \cite{Survey58,Dun-SchII,Dun-SchIII}. To any Borel measurable function $F:{\mathbb C}\to {\mathbb C}$ (or $F:\sigma(A)\to {\mathbb C}$, ${\mathbb C}$ is the \textit{complex plane}), there corresponds a scalar type spectral operator
\begin{equation*}
F(A):=\int_{\mathbb C} F(\lambda)\,dE_A(\lambda)
=\int_{\sigma(A)} F(\lambda)\,dE_A(\lambda)
\end{equation*}
defined as follows:
\begin{equation*}
\begin{split}
F(A)f&:=\lim_{n\to\infty}F_n(A)f,\quad f\in D(F(A)),\\
D(F(A))&:=\left\{f\in X
%\middle
\big| \lim_{n\to\infty}F_n(A)f\ \text{exists}\right\}
\end{split}
\end{equation*}
($D(\cdot)$ is the {\it domain} of an operator), where
\begin{equation*}
F_n(\cdot):=F(\cdot)\chi_{\{\lambda\in\sigma(A)\,|\,|F(\lambda)|\le
n\}}(\cdot), \quad  n\in\N,
\end{equation*}
($\chi_\delta(\cdot)$ is the {\it characteristic function} of a set $\delta\subseteq {\mathbb C}$, $\N:=\left\{1,2,3,\dots\right\}$ is the set of \textit{natural numbers}) and
\begin{equation*}
F_n(A):=\int_{\sigma(A)} F_n(\lambda)\,dE_A(\lambda),\quad
n\in\N,
\end{equation*}
are {\it bounded} scalar type spectral operators on $X$ defined in the same manner as for a {\it normal operator} (see, e.g., \cite{Dun-SchII,Plesner}).

In particular,
\begin{equation}\label{A}
A^n=\int_{{\mathbb C}} \lambda^n\,dE_A(\lambda)
=\int_{\sigma(A)} \lambda^n\,dE_A(\lambda),\quad n\in\Z_+,
\end{equation}
($\Z_+:=\left\{0,1,2,\dots\right\}$ is the set of \textit{nonnegative integers}).

%%%%%%%%%%%%%%%%%%%%%%%%%%%%%%%%%%%%%%%%%%%%%%%%%%%%%%%%%%
If a scalar type spectral operator $A$ generates $C_0$-semigroup
of linear operators, it is of the form
\begin{equation*}
e^{tA}=\int_{\mathbb C} e^{t\lambda}\,dE_A(\lambda)
=\int_{\sigma(A)} e^{t\lambda}\,dE_A(\lambda),\quad t\ge 0
\end{equation*}
\cite{Berkson1966,Markin2002(2),Panchapagesan1969}.

The properties of the {\it spectral measure}  $E_A(\cdot)$
and the {\it operational calculus}, exhaustively delineated in \cite{Survey58,Dun-SchIII}, underly the entire subsequent discourse. Here, we shall outline a few facts of particular importance.

Due to its {\it strong countable additivity}, the spectral measure $E_A(\cdot)$ is {\it bounded} \cite{Dun-SchI,Dun-SchIII}, i.e., there is such an $M>0$ that, for any Borel set $\delta\subseteq {\mathbb C}$,
\begin{equation}\label{bounded}
\|E_A(\delta)\|\le M.
\end{equation}
The notation $\|\cdot\|$ has been recycled here to designate the norm in the space $L(X)$ of all bounded linear operators on $X$. We shall adhere to this rather common economy of symbols in what follows adopting the same notation for the norm in the \textit{dual space} $X^*$ as well.

For any $f\in X$ and $g^*\in X^*$, the \textit{total variation} $v(f,g^*,\cdot)$ of the complex-valued Borel measure $\langle E_A(\cdot)f,g^* \rangle$
($\langle \cdot,\cdot \rangle$ is the {\it pairing} between the space $X$ and its dual $X^*$) is a {\it finite} positive Borel measure with
\begin{equation}\label{tv}
v(f,g^*,{\mathbb C})=v(f,g^*,\sigma(A))\le 4M\|f\|\|g^*\|
\end{equation}
(see, e.g., \cite{Markin2004(1),Markin2004(2)}).
Also (Ibid.), $F:{\mathbb C}\to {\mathbb C}$ (or $F:\sigma(A)\to {\mathbb C}$) being an arbitrary Borel measurable function, for any $f\in D(F(A))$, $g^*\in X^*$, and an arbitrary Borel set $\sigma\subseteq {\mathbb C}$,
\begin{equation}\label{cond(ii)}
\int_\sigma|F(\lambda)|\,dv(f,g^*,\lambda)
\le 4M\|E_A(\sigma)F(A)f\|\|g^*\|.
\end{equation}
In particular,
\begin{equation}\label{cond(i)}
\int_{{\mathbb C}}|F(\lambda)|\,d v(f,g^*,\lambda)
=\int_{\sigma(A)}|F(\lambda)|\,d v(f,g^*,\lambda)\le 4M\|F(A)f\|\|g^*\|.
\end{equation}
The constant $M>0$ in \eqref{tv}--\eqref{cond(i)} is from
\eqref{bounded}.

Subsequently, the frequent terms {\it "spectral measure"} and {\it "operational calculus"} will be abbreviated to {\it s.m.} and {\it o.c.}, respectively.

%%%%%%%%%%%%%%%%%%%%%%%%%%%%%%%%%%%%%%%%%%%%%%%%%%%%%%%%%%%%%
\subsection{The Carleman Classes of Functions}

Let $I$ be an interval of the real axis $\R$, $C^\infty(I,X)$ be the space of all $X$-valued functions
strongly infinite differentiable on $I$,
and $\left\{m_n\right\}_{n=0}^\infty$ be a sequence of positive numbers.

The subspaces of $C^\infty(I,X)$
\begin{multline*}
C_{\{m_n\}}(I, X):=\bigl\{g(\cdot)\in C^{\infty}(I, X) \bigm |
\forall [a,b] \subseteq I \ \exists \alpha>0\ \exists c>0:
\\
\hfill
\max_{a \le t \le b}\|g^{(n)}(t)\| \le c\alpha^nm_n,
\ n\in\Z_+\bigr\},
\\
\shoveleft{
C_{(m_n)}(I,X):=\bigl\{g(\cdot) \in C^{\infty}(I,X) \bigm |
\forall [a,b] \subseteq I\ \forall \alpha > 0 \ \exists c>0:
}\\
\hfill
\max_{a \le t \le b}\|g^{(n)}(t)\| \le c\alpha^nm_n,\ n\in\Z_+\bigr\}
\end{multline*}
are called the {\it Carleman classes} of strongly ultradifferentiable on $I$ vector functions corresponding to the sequence $\left\{m_n\right\}_{n=0}^\infty$
of {\it Roumieu} and {\it Beurling type}, respectively
(for scalar functions, see \cite{Carleman,Komatsu,Mandel}).

The inclusions
\begin{equation}\label{incl1}
C_{(m_n)}(I,X)\subseteq C_{\{m_n\}}(I,X)\subseteq C^{\infty}(I,X)
\end{equation}
are obvious.

If two sequences of positive numbers $\bigl\{m_n \bigr\}_{n=0}^\infty$ and $\bigl\{m'_n \bigr\}_{n=0}^\infty$ are related  as follows:
\begin{equation*}
\forall \gamma > 0 \ \exists c=c(\gamma)>0:\ m'_n \le c\gamma^n
m_n,\quad  n\in\Z_+,
\end{equation*}
we also have the inclusion
\begin{equation}\label{incl2}
C_{\{m'_n\}}(I,X) \subseteq C_{(m_n)}(I,X),
\end{equation}
the sequences being subject to the condition
\begin{equation*}
\exists \gamma_1,\gamma_2 > 0, \ \exists c_1,c_2>0:\ c_1\gamma_1^n
m_n\le m'_n \le c_2\gamma_2^n m_n,\quad n\in\Z_+,
\end{equation*}
their corresponding Carleman classes coincide:
\begin{equation}\label{equal}
C_{\{m_n\}}(I,X)=C_{\{m'_n\}}(I,X),\quad
C_{(m_n)}(I,X)=C_{(m'_n)}(I,X).
\end{equation}
Considering {\it Stirling's formula} and the latter,
\begin{equation*}
\begin{split}
&{\mathcal E}^{\{\beta\}}(I,X):=C_{\{[n!]^\beta\}}(I,X)
=C_{\{n^{\beta n}\}}(I,X),\\
&{\mathcal E}^{(\beta)}(I,X):=C_{([n!]^\beta)}(I,X)
=C_{(n^{\beta n})}(I,X)
\end{split}
\end{equation*}
with $\beta\ge 0$ are the well-known \textit{Gevrey classes} of strongly ultradifferentiable on $I$ vector functions of order $\beta$ of Roumieu and Beurling type, respectively (for scalar functions, see \cite{Gevrey}). In particular, ${\mathcal E}^{\{1\}}(I,X)$ and ${\mathcal E}^{(1)}(I,X)$ are the classes of {\it analytic} on $I$ and {\it entire} vector functions, respectively; ${\mathcal E}^{\{0\}}(I,X)$ and ${\mathcal E}^{(0)}(I,X)$ (i.e., the classes $C_{\{1\}}(I,X)$ and $C_{(1)}(I,X)$ corresponding to the sequence $m_n\equiv 1$) are the classes of \textit{entire} vector functions of \textit{exponential} and \textit{minimal exponential type}, respectively.

%%%%%%%%%%%%%%%%%%%%%%%%%%%%%%%%%%%%%%%%%%%%%%%%%%%%%%%%%%%%%
\subsection{The Carleman Classes of Vectors}

Let $A$ be a densely defined closed linear operator in a complex Banach space $(X,\|\cdot\|)$ and $\left\{m_n\right\}_{n=0}^\infty$ be a sequence of positive numbers and
\begin{equation*}
C^{\infty}(A):=\bigcap_{n=0}^{\infty}D(A^n).
\end{equation*}

The subspaces of $C^{\infty}(A)$
\begin{equation*}
\begin{split}
C_{\{m_n\}}(A)&:=\left\{f\in C^{\infty}(A) \big | \exists
\alpha>0\ \exists c>0: \|A^nf\| \le c\alpha^n m_n,\ n\in\Z_+
\right\},\\ C_{(m_n)}(A)&:=\left\{f \in C^{\infty}(A) \big  |
\forall \alpha > 0 \ \exists c>0: \|A^nf\| \le c\alpha^n m_n,\
n\in\Z_+ \right\}
\end{split}
\end{equation*}
are called the {\it Carleman classes} of ultradifferentiable
vectors of the operator $A$ corres\-ponding to the sequence
$\left\{m_n\right\}_{n=0}^\infty$ of {\it Roumieu} and {\it
Beurling type}, respectively.

For the Carleman classes of vectors, the inclusions analogous to \eqref{incl1} and \eqref{incl2} and the equalities analogous to \eqref{equal} are valid.

For $\beta\ge 0$,
\begin{equation*}
\begin{split}
&{\mathcal E}^{\{\beta\}}(A):=C_{\{[n!]^\beta\}}(A)
=C_{\{n^{\beta n}\}}(A),\\
&{\mathcal E}^{(\beta)}(A):=C_{([n!]^\beta)}(A)
=C_{(n^{\beta n})}(A)
\end{split}
\end{equation*}
are the well-known \textit{Gevrey classes} of strongly ultradifferentiable vectors of $A$ of order $\beta$ of Roumieu
and Beurling type, respectively (see, e.g., \cite{GorV83,book,Gor-Knyaz}). In particular,
${\mathcal E}^{\{1\}}(A)$ and ${\mathcal E}^{(1)}(A)$ are the well-known classes of {\it analytic} and
{\it entire} vectors of $A$, respectively \cite{Goodman,Nelson}; ${\mathcal E}^{\{0\}}(A)$ and
${\mathcal E}^{(0)}(A)$ (i.e., the classes $C_{\{1\}}(A)$ and $C_{(1)}(A)$ corresponding to the sequence
$m_n\equiv 1$) are the classes of \textit{entire} vectors of \textit{exponential} and \textit{minimal exponential
type}, respectively (see, e.g., \cite{Gor-Knyaz,Radyno1983(1)}).

%%%%%%%%%%%%%%%%%%%%%%%%%%%%%%%%%%%%%%%%%%%%%%%%%%%%%%%%%%%%%
\subsection{Conditions on the Sequence $\left\{m_n
\right\}_{n=0}^\infty$}

If a sequence of positive numbers \break
$\left\{m_n\right\}_{n=0}^\infty$ satisfies the condition
\begin{equation*}
\textbf{(WGR)}\ \forall \alpha>0\ \exists c=c(\alpha)>0:\
c\alpha^n \le m_n,\quad n\in\Z_+,
\end{equation*}
the scalar function
\begin{equation}\label{T}
T(\lambda):=m_0\sum_{n=0}^{\infty} \frac{\lambda^n}{m_n},\quad
\lambda\ge 0 \quad (0^0:=1)
\end{equation}
first introduced by S. Mandelbrojt \cite{Mandel}, is well-defined (cf. \cite{Gor-Knyaz}).
The function is {\it con\-tinuous}, {\it strictly increasing}, and $T(0)=1$.

Hence, the function
\begin{equation}\label{M}
M(\lambda):=\ln T(\lambda),\quad \lambda\ge 0,
\end{equation}
is {\it continuous}, {\it strictly increasing} and $M(0)=0$. Its {\it inverse} $M^{-1}(\cdot)$ is defined on $[0,\infty)$ and inherits all the aforementioned properties of $M(\cdot)$.

As is shown in \cite{GorV83} (see also \cite{Gor-Knyaz} and \cite{book}), the sequence $\left\{m_n\right\}_{n=0}^\infty$ satisfying the condition \textbf{(WGR)}, for a {\it normal operator} $A$ in a complex Hilbert space $X$, the equalities
\begin{equation}\label{CC}
\begin{split}
&C_{\{m_n\}}(A)=\bigcup_{t>0}D(T(t|A|)),\\
&C_{(m_n)}(A)=\bigcap_{t>0}D(T(t|A|))\\
\end{split}
\end{equation}
are true, the operators $T(t|A|)$, $t>0$, defined in the sense of the operational calculus for a normal operator (see, e.g., \cite{Dun-SchII,Plesner}) and the function $T(\cdot)$ being replaceable with any {\it nonnegative}, {\it continuous}, and {\it increasing} on $[0,\infty)$ function $F(\cdot)$ satisfying
\begin{equation}\label{replace}
c_1F(\gamma_1\lambda)\le T(\lambda)\le c_2F(\gamma_2\lambda),\quad
\lambda\ge R,
\end{equation}
with some $\gamma_1,\gamma_2,c_1,c_2>0$ and $R\ge 0$, in particular, with
\begin{equation*}
S(\lambda):=m_0\sup_{n\ge 0}\dfrac{\lambda^n}{m_n}, \quad
\lambda\ge 0, \quad \text{or}\quad
P(\lambda):=m_0\biggl[\sum_{n=0}^\infty\dfrac{\lambda^{2n}}{m_n^2}\biggr]^{1/2},
\quad \lambda\ge 0,
\end{equation*}
(cf. \cite{Gor-Knyaz}).

In {\cite[Theorem 3.1]{Markin2004(2)}}, the above is generalized to the case of a \textit{scalar type spectral operator} $A$ in a \textit{reflexive} complex Banach space $X$, the reflexivity requirement shown to be superfluous in {\cite[Theorem 3.1]{Markin2015(1)}}.

%%%%%%%%%%%%%%%%%%%%%%%%%%%%%%%%%%%%%%%%%%%%%%%%%%%%%%%%%%%%%
In \cite{Markin2008}, the sequence $\left\{m_n\right\}_{n=0}^\infty$ is subject to the following conditions:
\begin{equation*}
\textbf{(GR)}\ \exists \alpha>0\ \exists c>0:\ c\alpha^nn! \le
m_n,\quad n\in\Z_+,
\end{equation*}
and
\begin{equation*}
\textbf{(SBC)}\ \exists h,H>1\ \exists l,L>0:\ lh^n\le
\sum_{k=0}^n \frac{m_n}{m_km_{n-k}}\le LH^n,\quad n\in\Z_+.
\end{equation*}

%%%%%%%%%%%%%%%%%%%%%%%%%%%%%%%%%%%%%%%%%%%%%%%%%%%%%%%%%%%%%
The former is a stronger version of {\bf (WGR)}, both {\bf (WGR)}
and {\bf (GR)} being restrictions on the \textit{growth} of
$\left\{m_n\right\}_{n=0}^\infty$, which explains the names. The
latter resembles the fundamental property of the {\it binomial
coefficients}
\begin{equation*}
\sum_{k=0}^n\binom{n}{k}=2^n,\quad  n\in\Z_+,
\end{equation*}
which also explains the name, and is precisely arrived at for $m_n=n!$.

Both {\bf (GR)} and {\bf (SBC)} are satisfied for
$m_n=[n!]^\beta$ with $\beta\ge 1$ (see \cite{Markin2008} for details).

Here, the sequence $\left\{m_n\right\}_{n=0}^\infty$ will be
subject to a stronger version of {\bf (GR)}
\begin{equation*}
\textbf{(SGR)}\ \forall \alpha>0\ \exists c=c(\alpha)>0:\
c\alpha^nn! \le m_n,\quad n\in\Z_+,
\end{equation*}
and a weaker version of {\bf (SBC)}
\begin{equation*}
\textbf{(BC)}\ \exists h>1\ \exists l>0:\ lh^n\le \sum_{k=0}^n
\frac{m_n}{m_km_{n-k}},\quad n\in\Z_+.
\end{equation*}

Both {\bf (SGR)} and {\bf (BC)} are satisfied for
$m_n=[n!]^\beta$ with $\beta>1$, also for $m_n=e^{n^2}$ (see \cite{Markin2009} for details).

Observe that there are examples demonstrating the
independence of the conditions {\bf (GR)} and {\bf (BC)} \cite{Markin2009} (cf. \cite{Markin2008}).

As is shown in \cite{Markin2008}, the conditions {\bf (GR)} and {\bf (SBC)} have the following implications for the function $M(\cdot)$ defined in \eqref{M} and its inverse $M^{-1}(\cdot)$:
\begin{equation*}\label{GR}
\exists \alpha>0\ \exists R>0:\
2\alpha^{-1}M^{-1}(\lambda)\ge \lambda,\quad \lambda\ge M(R),
\end{equation*}
and
\begin{equation*}\label{SBC}
%\begin{split}
\begin{gathered}
%&
\exists h,H>1\ \exists l,L>0\ \text{(the constants from the condition \textbf{(SBC)})}:\\
%&
2^{-n}M(h^n\lambda)+[1-2^{-n}]\ln(m_0l)\le  M(\lambda)
%\\
%&
\le 2^{-n}M(H^n\lambda)+[1-2^{-n}]\ln(m_0L),\\
%\quad
n\in\N,\quad \lambda\ge 0.
%\end{split}
\end{gathered}
\end{equation*}

Observe that, from \textbf{(SBC)} with $n=0$, the estimates
\begin{equation*}
\ln(m_0l)\le 0\le \ln(m_0L)
\end{equation*}
are inferred immediately.

The conditions {\bf (SGR)} and {\bf (BC)} imply
\begin{equation}\label{SGR}
\forall \alpha>0\ \exists R=R(\alpha)>0:\
2\alpha^{-1}M^{-1}(\lambda)\ge \lambda,\quad \lambda\ge M(R),
\end{equation}
and
\begin{equation}\label{BC1}
\begin{gathered}
%\begin{split}
%&
\exists h>1\ \exists l>0\ \text{(the constants of the condition \textbf{(BC)})}:\\
%&
2^{-n}M(h^n\lambda)+[1-2^{-n}]\ln(m_0l)\le  M(\lambda),\quad n\in\N,\quad \lambda\ge 0
%\end{split}
\end{gathered}
\end{equation}
(see \cite{Markin2009} for details).

Substituting $h^{-n}\lambda$ for $\lambda$, we obtain the following equivalent version:
\begin{equation}\label{BC2}
\begin{split}
& \exists h>1\ \exists l>0\ \text{(the constants from the condition \textbf{(BC)})}:\\
& M(\lambda)\le 2^{n}M(h^{-n}\lambda)-[2^{n}-1]\ln(m_0l),\quad \lambda\ge 0,\quad n\in\N.
\end{split}
\end{equation}

%%%%%%%%%%%%%%%%%%%%%%%%%%%%%%%%%%%%%%%%%%%%%%%%%%%%%%%%%%%%%
\section{Beurling type Carleman ultradifferentiable $C_0$-semigroups}

\begin{defn}\label{def1}
Let $\left\{m_n\right\}_{n=0}^{\infty}$ be a sequence of positive numbers. We shall call a $C_0$-semigroup
$\left\{S(t)\big|t\ge0\right\}$ in a complex Banach space $(X,\|\cdot\|)$ a Roumieu (Beurling)
type Carleman ultradifferentiable $C_0$-semigroup corresponding to the sequence $\left\{m_n\right\}_{n=0}^{\infty}$,
or a $C_{\{m_n\}}$-semigroup (a $C_{(m_n)}$-semigroup), if each orbit $S(\cdot)f$, $f\in X$, belongs
to the Roumieu (Beurling) type Carleman class of vector functions
\begin{equation*}
C_{\{m_n\}}\left((0,\infty),X\right)\quad
(C_{(m_n)}\left((0,\infty),X\right),\ \text{respectively})
\end{equation*}
(cf. \cite{Markin2008}).
\end{defn}

Recall that in \cite{Markin2008}, we have proved the following statements.

\begin{prop}
%[{
{\rm(}\cite[Proposition $4.1$]{Markin2008}{\rm)}.
%}]
\label{prop}
Let $A$ be a scalar type spectral operator in a complex Banach space $(X,\|\cdot\|)$ generating a
$C_0$-semigroup $\left\{e^{tA}\big|t\ge 0\right\}$ and $\left\{m_n\right\}_{n=0}^\infty$
be a sequence of positive numbers. Then the restriction of an orbit $e^{tA}f$, $t\ge 0$, $f\in X$,
to a subinterval $I\subseteq [0,\infty)$ belongs to the
Carleman class
$C_{\{m_n\}}(I,X)$ ($C_{(m_n)}(I,X)$) iff
\begin{equation*}
e^{tA}f \in C_{\{m_n\}}(A)
\ \text{($C_{(m_n)}(A)$, respectively)},\quad t\in I.
\end{equation*}
\end{prop}

\begin{thm}{\rm(}\cite[Theorem
$5.1$]{Markin2008}{\rm)}.\label{thm}
Let $\left\{m_n\right\}_{n=0}^\infty$ be a sequence of positive numbers satisfying the conditions {\bf(GR)} and {\bf(SBC)}. Then a scalar type spectral operator $A$ in a complex Banach space $(X,\|\cdot\|)$ generates a $C_{\{m_n\}}$-semigroup iff
there are $b>0$ and $a\in \R$ such that
\begin{equation*}
\Rep\lambda\le a-bM(|\Imp\lambda|),\quad \lambda\in\sigma(A),
\end{equation*}
where $M(\lambda)=\ln T(\lambda)$, $0\le\lambda<\infty$, and the function $T(\cdot)$ defined by \eqref{T} is
replaceable with any {\it nonnegative}, {\it continuous}, and {\it increasing} on $[0,\infty)$ function $F(\cdot)$
satis\-fying \eqref{replace}.
\end{thm}

Now, we are going to prove the following

\begin{thm}\label{GBTCUS}
Let $\left\{m_n\right\}_{n=0}^\infty$ be a sequence of positive numbers satisfying the conditions {\bf(SGR)} and {\bf(BC)}. Then a scalar type spectral operator $A$ in a complex Banach space $(X,\|\cdot\|)$ generates a $C_{(m_n)}$-semigroup iff, for any $b>0$, there is an $a\in \R$ such that
\begin{equation*}
\Rep\lambda\le a-bM(|\Imp\lambda|),\ \lambda\in\sigma(A),
\end{equation*}
where $M(\lambda)=\ln T(\lambda)$, $0\le\lambda<\infty$, and the function $T(\cdot)$ defined by \eqref{T} is replaceable with any {\it nonnegative}, {\it continuous}, and {\it increasing} on $[0,\infty)$ function $F(\cdot)$ satisfying \eqref{replace}.
\end{thm}

\begin{proof}
{{\bf "If"} Part.}\quad By the hypothesis,
\begin{equation*}
\Rep\lambda\le a,\quad  \lambda\in\sigma(A),
\end{equation*}
with some $a\in\R$, which, by {\cite[Proposition $3.1$]{Markin2002(2)}} implies that $A$ does generate a
$C_0$-semigroup of its exponentials $\left\{e^{tA}\big|t\ge 0\right\}$ (see \cite{Markin2002(2)},
cf. also \cite{Berkson1966,Panchapagesan1969}).

Consider an arbitrary \textit{orbit} $e^{tA}f$, $t\ge 0$, $f\in X$.

By Proposition \ref{prop}, we are to show that
\begin{equation*}
e^{tA}f\in C_{(m_n)}(A),\quad t>0.
\end{equation*}

For arbitrary $t>0$ and $s>0$, let us fix a sufficiently large $N\in \N$ so that
\begin{equation*}
h^{-N}2s\le 1,
\end{equation*}
where $h>1$ is the constant from the condition \textbf{(BC)},
and set
\begin{equation*}
b:=2^{N+1}t^{-1}>0.
\end{equation*}

Since, due to the condition \textbf{(SGR)}, $\alpha>0$ in \eqref{SGR} is arbitrary, we can assume that $\alpha:=b>0$.

For any $g^*\in X^*$,
$$
\begin{aligned}
%\begin{multline*}
& \int_{\sigma(A)}T(s|\lambda|)e^{t\Rep\lambda}\,dv(f,g^*,\lambda)
\\
& \qquad =\int_{\{\lambda\in\sigma(A)|\Rep\lambda\le\min(-2^{-1}bM(R),a)\}}T(s|\lambda|)
e^{t\Rep\lambda}\,dv(f,g^*,\lambda)
\\
& \qquad +\int_{\{\lambda\in\sigma(A)|\min(-2^{-1}bM(R),a)<\Rep\lambda\le a\}}T(s|\lambda|)
e^{t\Rep\lambda}\,dv(f,g^*,\lambda)
<\infty,
%\end{multline*}
\end{aligned}
$$
where $R=R(\alpha)>0$ is the constant from \eqref{SGR}.

Indeed, the latter of the two integrals in the right side of the equality is finite due to the \textit{boundedness} of the set
$\bigl\{\lambda\in\sigma(A)\bigm|\min(-2^{-1}bM(R),a)<\Rep\lambda\le a\bigr\}$ (for $a\le -2^{-1}bM(R)$, the set is, obviously, empty), the \textit{continuity} of the integrand on ${\mathbb C}$,
and the \textit{finiteness} of the measure $v(f,g^*,\cdot)$ (see \eqref{tv}).

For the former one, there are the two possibilities
\begin{equation*}
a\le 0\quad \text{or}\quad a>0.
\end{equation*}

If $a\le 0$,
\begin{multline}\label{m1}
\int_{\{\lambda\in\sigma(A)|\Rep\lambda\le\min(-2^{-1}bM(R),a)\}}T(s|\lambda|)
e^{t\Rep\lambda}\,dv(f,g^*,\lambda)
\\
\shoveleft{
=\int_{\{\lambda\in\sigma(A)|\Rep\lambda\le\min(-2^{-1}bM(R),a)\}}e^{M(s|\lambda|)}
e^{t\Rep\lambda}\,dv(f,g^*,\lambda)
}\\
\shoveleft{
\le\int_{\{\lambda\in\sigma(A)|\Rep\lambda\le\min(-2^{-1}bM(R),a)\}}e^{M(s[|\Rep\lambda|+|\Imp\lambda|])}
e^{t\Rep\lambda}\,dv(f,g^*,\lambda)
}\\
\hfill
\text{for}\ \lambda\in \sigma(A)\ \text{with}\ \Rep\lambda\le\min(-2^{-1}bM(R),a),
\\
\hfill
\Rep\lambda \le -2^{-1}bM(R)\le 0\ \text{and}\ |\Imp\lambda|\le M^{-1}(b^{-1}[a-\Rep\lambda]);
\\
\shoveleft{
\le \int_{\{\lambda\in\sigma(A)|\Rep\lambda\le\min(-2^{-1}bM(R),a)\}}\!\!\!\!\!
e^{M(s[-\Rep\lambda+M^{-1}(b^{-1}[a-\Rep\lambda])])}e^{t\Rep\lambda}\,dv(f,g^*,\lambda)
}\\
\hfill
\text{
since $a\le 0$, $a-\Rep\lambda\le -2\Rep\lambda$ whenever  $\Rep\lambda\le\min(-2^{-1}bM(R),a)\le 0$;
}
\\
\shoveleft{
\le \int_{\{\lambda\in\sigma(A)|\Rep\lambda\le\min(-2^{-1}bM(R),a)\}}
\!\!\!\!\!\!e^{M(s[-\Rep\lambda+M^{-1}(2b^{-1}[-\Rep\lambda])])}
e^{t\Rep\lambda}
\,dv(f,g^*,\lambda)
}\\
\hfill
\text{by \eqref{SGR}, $2b^{-1}[-\Rep\lambda]\le 2\alpha^{-1}M(2b^{-1}[-\Rep\lambda])$ whenever $\Rep\lambda\le-2^{-1}bM(R)$;}
\\
\hfill
\text{since $\alpha:=b$, $-\Rep\lambda\le M(2b^{-1}[-\Rep\lambda])$ whenever $\Rep\lambda\le-2^{-1}bM(R)$;}
\\
\shoveleft{
\le \int_{\{\lambda\in\sigma(A)|\Rep\lambda\le\min(-2^{-1}bM(R),a)\}}
e^{M(2sM^{-1}(2b^{-1}[-\Rep\lambda]))}
e^{t\Rep\lambda}\,dv(f,g^*,\lambda)
}\\
\hfill
\text{by \eqref{BC2};} 
\end{multline}	
\begin{multline*}
= \int_{\{\lambda\in\sigma(A)|\Rep\lambda\le\min(-2^{-1}bM(R),a)\}}
e^{2^NM(h^{-N}2sM^{-1}(2b^{-1}[-\Rep\lambda]))-[2^N-1]\ln(m_0l)}
\\
\hfill
\times e^{t\Rep\lambda}\,dv(f,g^*,\lambda)
\\
\shoveleft{
=(m_0L)^{-[2^{N}-1]}
\int_{\{\lambda\in\sigma(A)|\Rep\lambda\le\min(-2^{-1}bM(R),a)\}}
e^{2^NM(h^{-N}2sM^{-1}(2b^{-1}[-\Rep\lambda]))}
}\\
\hfill
\times e^{t\Rep\lambda}\,dv(f,g^*,\lambda)
\\
\hfill
\text{by choice, $h^{-N}2s\le 1$;}
\\
\shoveleft{
\le (m_0L)^{-[2^{N}-1]}
\int_{\{\lambda\in\sigma(A)|\Rep\lambda\le\min(-2^{-1}bM(R),a)\}}
e^{2^NM(M^{-1}(2b^{-1}[-\Rep\lambda]))}
}\\
\hfill
\times e^{t\Rep\lambda}\,dv(f,g^*,\lambda)
\\
\shoveleft{
=(m_0L)^{-[2^{N}-1]}
\int_{\{\lambda\in\sigma(A)|\Rep\lambda\le\min(-2^{-1}bM(R),a)\}}
e^{2^{N+1}b^{-1}[-\Rep\lambda]+t\Rep\lambda}
\,dv(f,g^*,\lambda)
}\\
\hfill
\text{by choice, $b:=2^{N+1}t^{-1}$;}
\\
\shoveleft{
= (m_0L)^{-[2^{N}-1]}v\left(f,g^*,\left\{\lambda\in\sigma(A)\big|\Rep\lambda\le\min(-2^{-1}bM(R),a)\right\}\right)
}\\
\shoveleft{
\le (m_0L)^{-[2^{N}-1]}v(f,g^*,\sigma(A))
}\\
\hfill
\text{by \eqref{tv};}
\\
\le (m_0L)^{-[2^{N}-1]}4M\|f\|\|g^*\|<\infty.
\end{multline*}

If $a>0$,
\begin{multline*}
\int_{\{\lambda\in\sigma(A)|\Rep\lambda\le\min(-2^{-1}bM(R),a)\}}T(s|\lambda|)
e^{t\Rep\lambda}\,dv(f,g^*,\lambda)
\\
\shoveleft{
=\int_{\{\lambda\in\sigma(A)|\Rep\lambda\le\min(-2^{-1}bM(R),-a)\}}T(s|\lambda|)
e^{t\Rep\lambda}
e^{t\Rep\lambda}\,dv(f,g^*,\lambda)
}\\
+\int_{\{\lambda\in\sigma(A)|\min(-2^{-1}bM(R),-a)<\Rep\lambda\le -2^{-1}bM(R)\}}T(s|\lambda|)
e^{t\Rep\lambda}\,dv(f,g^*,\lambda)<\infty.
\end{multline*}

Indeed, the latter of the two integrals in the right side of the equality is finite due to the \textit{boundedness} of the set
$\bigl\{\lambda\in\sigma(A) \bigm|\min(-a,-2^{-1}bM(R))<\Rep\lambda\le -2^{-1}bM(R)\bigr\}$
(for $a\le 2^{-1}bM(R)$, the set is, obviously, empty), the \textit{continuity} of the integrand on ${\mathbb C}$, and the \textit{finiteness} of the measure $v(f,g^*,\cdot)$ (see \eqref{tv}).

For the former one, we have
\begin{multline*}
\int_{\{\lambda\in\sigma(A)|\Rep\lambda\le\min(-2^{-1}bM(R),-a)\}}T(s|\lambda|)
e^{t\Rep\lambda}\,dv(f,g^*,\lambda)
\\
\shoveleft{
=\int_{\{\lambda\in\sigma(A)|\Rep\lambda\le\min(-2^{-1}bM(R),-a)\}}e^{M(s|\lambda|)}
e^{t\Rep\lambda}\,dv(f,g^*,\lambda)
}\\
\shoveleft{
\le\int_{\{\lambda\in\sigma(A)|\Rep\lambda\le\min(-2^{-1}bM(R),-a)\}}e^{M(s[|\Rep\lambda|+|\Imp\lambda|])}
e^{t\Rep\lambda}\,dv(f,g^*,\lambda)
}\\
\hfill
\text{for}\ \lambda\in \sigma(A)\ \text{with}\ \Rep\lambda\le\min(-2^{-1}bM(R),-a),
\\
\hfill
\Rep\lambda \le -2^{-1}bM(R)\le 0\ \text{and}\ |\Imp\lambda|\le M^{-1}(b^{-1}[a-\Rep\lambda]);
\\
\shoveleft{
\int_{\{\lambda\in\sigma(A)|\Rep\lambda\le \min(-2^{-1}bM(R),-a)\}}
e^{M(s[-\Rep\lambda+M^{-1}(b^{-1}[a-\Rep\lambda])])}
e^{t\Rep\lambda}\,dv(f,g^*,\lambda)
}\\
\hfill
\text{$a-\Rep\lambda \le -2\Rep\lambda$ whenever $\Rep\lambda \le -a$;}
\end{multline*}
\begin{multline*}
\le \int_{\{\lambda\in\sigma(A)|\Rep\lambda\le\min(-2^{-1}bM(R),-a)\}}
e^{M(s[-\Rep\lambda+M^{-1}(2b^{-1}[-\Rep\lambda])])}
e^{t\Rep\lambda}\,dv(f,g^*,\lambda)
\\
\hfill
\text{in the same manner as in \eqref{m1};}
\\
<\infty.
\end{multline*}

Thus, we have proved that, for arbitrary $s>0$, $t>0$, $f\in X$, and $g^*\in X^*$,
\begin{equation}\label{ci}
\int_{\sigma(A)} T(s|\lambda|)e^{t\Rep\lambda}\,dv(f,g^*,\lambda)<\infty.
\end{equation}
Furthermore, for any $s>0$, $t>0$, $f\in X$,
\begin{equation*}\label{cii}
\sup_{\{g^*\in X^*\,|\,\|g^*\|=1\}}\int_{\{\lambda\in\sigma(A)| T(s|\lambda|)
e^{t\Rep\lambda}>n\}}
T(s|\lambda|)e^{t\Rep\lambda}\,dv(f,g^*,\lambda)
\to 0\quad \text{as\quad $n\to\infty$}.
\end{equation*}
Indeed, as follows from the preceding argument,
for any $s,t>0$, the spectrum $\sigma(A)$ can be partitioned into two Borel subsets $\sigma_1$ and $\sigma_2$ ($\sigma(A)=\sigma_1\cup\sigma_2$, $\sigma_1\cap\sigma_2=\emptyset$) in such a way that $\sigma_1$ is \textit{bounded} and
\begin{equation*}
T(s|\lambda|)e^{t\Rep\lambda}\le 1,\quad \lambda\in\sigma_2.
\end{equation*}

Therefore,
\begin{multline*}
\sup_{\{g^*\in X^*\,|\,\|g^*\|=1\}}\int_{\{\lambda\in\sigma(A)|
T(s|\lambda|)e^{t\Rep\lambda}>n\}}
T(s|\lambda|)e^{t\Rep\lambda}\,dv(f,g^*,\lambda)
\\
\shoveleft{
= \sup_{\{g^*\in X^*\,|\,\|g^*\|=1\}}\biggl[\int_{\{\lambda\in\sigma_1|
T(s|\lambda|)e^{t\Rep\lambda}>n\}}
T(s|\lambda|)e^{t\Rep\lambda}\,dv(f,g^*,\lambda)
}\\
\shoveleft{
+\int_{\{\lambda\in\sigma_2| T(s|\lambda|)
e^{t\Rep\lambda}>n\}}
T(s|\lambda|)e^{t\Rep\lambda}\,dv(f,g^*,\lambda)\biggr]
}\\
\hfill
\text{since $\sigma_1$ is bounded and $T(s|\cdot|)e^{t\Rep\cdot}$ is continuous on ${\mathbb C}$, }
\\
\hfill
\text{there is such a $C\ge 1$ that}\
T(s|\lambda|)e^{t\Rep\lambda}\le C,\ \lambda\in\sigma_1;
\\
\shoveleft{
\le \sup_{\{g^*\in X^*\,|\,\|g^*\|=1\}}\bigl[Cv\left(f,g^*,\{\lambda\in\sigma_1\big| T(s|\lambda|)e^{t\Rep\lambda}>n\}\right)
}\\
\shoveleft{
+ v\left(f,g^*,\{\lambda\in\sigma_2\big| T(s|\lambda|)e^{t\Rep\lambda}>n\}\right)\bigr]
}\\
\shoveleft{
\le \sup_{\{g^*\in X^*\,|\,\|g^*\|=1\}}Cv\left(f,g^*,\{\lambda\in\sigma(A)\big| T(s|\lambda|)e^{t\Rep\lambda}>n\}\right)
\hfill
\text{by \eqref{cond(ii)} with $F(\lambda)\equiv 1$;}
}\\
\shoveleft{
\le \sup_{\{g^*\in X^*\,|\,\|g^*\|=1\}}C4M
\|E_A(\{\lambda\in\sigma(A)| T(s|\lambda|)e^{t\Rep\lambda}>n\})f\|\|g^*\|
}\\
\shoveleft{
= 4CM\|E_A(\{\lambda\in\sigma(A)| T(s|\lambda|)e^{t\Rep\lambda}>n\})f\|
\hfill
\text{by the strong continuity of the {\it s.m.};}
}\\
\to 0\quad \text{as\quad $n\to\infty$.}
\end{multline*}

According to {\cite[Proposition $3.1$]{Markin2002(1)}}, \eqref{ci} and \eqref{cii} imply that,
for any $t>0$, $f\in X$, and $s>0$,
\begin{equation*}
e^{tA}f\in D(T(s|A|)).
\end{equation*}
Hence, for any $f\in X$, due to \eqref{CC},
\begin{equation*}
e^{tA}f\in \bigcap_{s>0} D(T(s|A|))=C_{(m_n)}(A),\quad t>0,
\end{equation*}
which, by Proposition \ref{prop}, implies that,
for $f\in X$,
\begin{equation*}
e^{\cdot A}f\in C_{(m_n)}((0,\infty),X),
\end{equation*}
i.e., the $C_0$-semigroup $\left\{e^{tA}\big|t\ge 0\right\}$ generated by $A$ is a $C_{(m_n)}$-semigroup.

%%%%%%%%%%%%%%%%%%%%%%%%%%%%%%%%%%%%%%%%%%%%%%%%%%%%%%%%%%%%%
\medskip
{{\bf "Only if"} Part.}\quad We shall prove this part by \textit{contrapositive}, i.e., assuming that there is such a $b>0$ that for any $a\in\R$,
\begin{equation*}
\sigma(A)\setminus\bigl\{\lambda\in {\mathbb C} \bigm|\Rep\lambda\le
a-bM(|\Imp\lambda|)\bigr\}\neq\emptyset,
\end{equation*}
we are to show that $A$ does not generate a $C_{(m_n)}$-semigroup.

Observe that the latter readily implies that the set
\begin{equation*}
\sigma(A)\setminus\left\{\lambda\in {\mathbb C} \big|\Rep\lambda\le
-bM(|\Imp\lambda|)
\right\}
\end{equation*}
is {\it unbounded},

For $\sigma(A)$, there are two possibilities
\begin{equation*}
\sup_{\lambda\in\sigma(A)}\Rep\lambda=\infty\quad \text{or}
\quad \sup_{\lambda\in\sigma(A)}\Rep\lambda<\infty.
\end{equation*}

The first one implies that $A$ does not generate
a $C_0$-semigroup \cite{Hille-Phillips}, let alone a $C_{(m_n)}$-semigroup.

With
\begin{equation}\label{upbound}
\sup_{\lambda\in\sigma(A)}\Rep\lambda<\infty
\end{equation}
being the case, $A$ generates a $C_0$-semigroup of its exponentials $\left\{e^{tA}\big|t\ge 0\right\}$  \cite{Markin2002(2)} and one can choose a sequence of points $\left\{\lambda_n\right\}_{n=1}^\infty$ in the complex plane as follows:
\begin{equation*}
\begin{gathered}
%&
\lambda_n \in \sigma(A),\quad  n\in\N,\\
%&
\Rep\lambda_n>-bM(|\Imp\lambda_n|),\quad  n\in\N,\quad \text{and}\\
%&
\lambda_0:=0,\quad |\lambda_n|>\max\bigl[n,|\lambda_{n-1}|\bigr],\quad n\in\N.
\end{gathered}
\end{equation*}
The latter, in particular, indicates that the points $\lambda_n$ are {\it
distinct}
\begin{equation*}
\lambda_i \neq \lambda_j,\quad i\neq j.
\end{equation*}
Since each set
\begin{equation*}
\left\{ \lambda \in {\mathbb C} \big| \Rep\lambda > -bM(|\Imp\lambda|),
\ |\lambda|>\max\left[n,|\lambda_{n-1}|\right] \right\},\quad n\in\N,
\end{equation*}
is {\it open} in ${\mathbb C}$, there exists such an $\varepsilon_n>0$ that, along with the point $\lambda_n$, the set contains the {\it open disk}
\begin{equation*}
\Delta_n=\left\{ \lambda \in {\mathbb C} \big||\lambda-\lambda_n|<\varepsilon_n \right\},
\end{equation*}
i.e., for any $\lambda \in \Delta_n$,
\begin{equation}\label{disks1}
\Rep\lambda>-bM(|\Imp\lambda|)\quad \text{and}\quad |\lambda|>\max\bigl[n,|\lambda_{n-1}|\bigr].
\end{equation}
The radii of the disks $\varepsilon_n$ can be chosen small enough so that
\begin{equation}\label{radii1}
0<\varepsilon_n<1/n,\quad n\in\N,\quad
\text{and}\quad \Delta_i \cap \Delta_j=\emptyset,\quad  i\neq j,
\end{equation}
i.e., the disks are {\it pairwise disjoint}.

Considering that each $\Delta_n \cap \sigma(A)\neq \emptyset$, $\Delta_n$ being an {\it open set}, by the properties of the {\it s.m.} and the latter, we infer
\begin{equation}\label{perp1}
E_A(\Delta_n)\neq 0,\quad n\in\N,\quad
\text{and}\quad E_A(\Delta_i)E_A(\Delta_j)=\delta_{ij}E_A(\Delta_i),
\end{equation}
($\delta_{ij}$ is {\it Kronecker's delta symbol} and $0$, here and whenever appropriate, designates the \textit{zero operator}).
Hence, the subspaces $E_A(\Delta_n)X$ are \textit{nontrivial} and
\begin{equation*}
E_A(\Delta_i)X\cap E_A(\Delta_j)X=\left\{0\right\},\quad i\neq j.
\end{equation*}
Thus, choosing vectors
\begin{equation}\label{vecseq}
e_n\in E_A(\Delta_n)X,\quad n\in\N,\quad \text{with}\quad \|e_n\|=1,
\end{equation}
we obtain a vector sequence $\left\{e_n\right\}_{n=1}^\infty$ such that, by \eqref{perp1},
\begin{equation}\label{ortho1}
E_A(\Delta_i)e_j=\delta_{ij}e_i.
\end{equation}
The latter, showing the \textit{linear independence} of $\left\{e_1,e_2,\dots\right\}$, goes a step beyond implying the existence of an $\varepsilon>0$ such that
\begin{equation}\label{dist1}
d_n:=\dist\left(e_n,\spa\left(\left\{e_i\big|i\in\N,\ i\neq n\right\}\right)\right)\ge \varepsilon,\quad n\in\N.
\end{equation}
Otherwise, there is a vanishing subsequence
$\left\{d_{n(k)}\right\}_{k=1}^\infty$
\begin{equation*}
d_{n(k)}\to 0\quad \text{as}\quad k\to\infty,
\end{equation*}
and hence, for any $k\in\N$, there exists an
\begin{equation*}
f_{n(k)}\in \spa\left(\left\{e_i\big|i\in\N,\ i\neq n(k)\right\}\right)\quad
\text{with}\quad
\|e_{n(k)}-f_{n(k)}\|<d_{n(k)}+1/n(k),
\end{equation*}
which, considering \eqref{bounded}, implies
\begin{equation*}
e_{n(k)}=E_A(\Delta_{n(k)})[e_{n(k)}-f_{n(k)}]\to 0\quad \text{as}\quad k\to\infty
\end{equation*}
contradicting \eqref{vecseq}.

As follows from the {\it Hahn-Banach Theorem} (see, e.g., \cite{Dun-SchI}), \eqref{dist1} implies that,
for each $n\in\N$, there is an $e^*_n\in X^*$ such that
\begin{equation}\label{H-B}
\|e_n^*\|=1\quad\text{and}\quad \langle e_i,e_j^*\rangle=\delta_{ij}d_i.
\end{equation}
For the sequence of the real parts $\{\Rep\lambda_n\}_{n=1}^\infty$, there are the two
possibilities
\begin{equation*}
\sup_{n\in \N}|\Rep\lambda_n|<\infty\quad\text{or}\quad
\sup_{n\in \N}|\Rep\lambda_n|=\infty.
\end{equation*}
Suppose that
\begin{equation}\label{bounded1}
\sup_{n\in \N}|\Rep\lambda_n|=:\omega<\infty.
\end{equation}
Let
\begin{equation*}
f:=\sum_{n=1}^\infty \dfrac{1}{n^2}e_{n}\in X
\quad \text{and}\quad
g^*:=\sum_{n=1}^\infty \dfrac{1}{n^2}e_{n}^*\in X^*,
\end{equation*}
the series strongly converging in $X$ and $X^*$, respectively, due to \eqref{vecseq} and \eqref{H-B}.
By \eqref{H-B} and \eqref{dist1},
\begin{equation}\label{dist2}
\langle e_n,g^*\rangle=\dfrac{1}{n^2}\langle e_n,e_{n}^*\rangle=\dfrac{d_n}{n^2}
\ge \dfrac{\varepsilon}{n^2},\quad n\in\N.
\end{equation}
As can be easily deduced from \eqref{ortho1},
\begin{equation}\label{vectors1}
E_A(\Delta_n)f=\frac{1}{n^2}e_n,\quad n\in\N,\quad \text{and}\quad E_A(\cup_{n=1}^\infty\Delta_n)f=f.
\end{equation}
Considering the latter and \eqref{dist2},
\begin{equation}\label{vt}
v(f,g^*,\Delta_n)\ge |\langle E_A(\Delta_n)f,g^* \rangle|
=\left\langle \frac{1}{n^2}e_n,g^* \right\rangle
\ge\dfrac{\varepsilon}{n^4},\quad n\in\N.
\end{equation}
For $s=t=1$, we have
\begin{multline}\label{similar}
\int_{\sigma(A)} T(|\lambda|)e^{\Rep\lambda}\,dv(f,g^*,\lambda)
\hfill \text{by \eqref{vectors1};}
\\
\shoveleft{
=
\int_{\sigma(A)} T(|\lambda|)e^{\Rep\lambda}
\,d v(E_A(\cup_{n=1}^\infty\Delta_n)f,g^*,\lambda)
\hfill \text{by the properties of the {\it o.c.};}
}\\
\shoveleft{
=\int_{\cup_{n=1}^\infty\Delta_n}T(|\lambda|)e^{\Rep\lambda}\,dv(f,g^*,\lambda)
=\sum_{n=1}^\infty\int_{\Delta_n}
T(|\lambda|)e^{\Rep\lambda}\,dv(f,g^*,\lambda)
}\\
\hfill
\text{for $\lambda\in \Delta_n$,
by \eqref{disks1}, \eqref{bounded1}, and \eqref{radii1}:}\quad
|\lambda|\ge n\ \text{and}
\\
\hfill
\Rep\lambda=\Rep\lambda_n
-(\Rep\lambda_n-\Rep\lambda)
\ge \Rep\lambda_n-|\lambda_n-\lambda|
\ge-\omega-\varepsilon_n\ge-\omega-1;
\end{multline}
\begin{multline*}	
\ge \sum_{n=1}^\infty T(n)e^{-(\omega+1)}v(f,g^*,\Delta_n)
\hfill
\text{by \eqref{vt};}
\\
\ge
e^{-(\omega+1)}\sum_{n=1}^\infty T(n)\frac{\varepsilon}{n^4}=\infty.
\end{multline*}
Indeed, by definition \eqref{T},
\begin{equation*}
T(n)\ge m_0\dfrac{n^4}{m_4},\quad n\in\N.
\end{equation*}
Hence, by {\cite[Proposition $3.1$]{Markin2002(1)}},
\begin{equation*}
e^{tA}f\bigr|_{t=1}\not\in D(T(|A|)).
\end{equation*}
Considering \eqref{CC}, the more so,
\begin{equation*}
e^{tA}f\bigr|_{t=1}\not\in \bigcap_{s>0}D(T(s|A|))= C_{(m_n)}(A).
\end{equation*}
Hence, according to Proposition \ref{prop},
\begin{equation*}
e^{\cdot A}f\not\in C_{(m_n)}\left((0,\infty),X\right),
\end{equation*} 
which implies that the $C_0$-semigroup $\left\{e^{tA}\big|t\ge 0\right\}$ generated by $A$ is not a $C_{(m_n)}$-semi\-group.
Suppose that
\begin{equation*}
\sup_{n\in \N}|\Rep\lambda_n|=\infty
\end{equation*}
and recall that we are also acting under hypothesis \eqref{upbound}.
Hence, there is a sub\eqref{vecseq} and \eqref{H-B}sequence
$\left\{\Rep\lambda_{n(k)}\right\}_{k=1}^\infty$ such that
\begin{equation}\label{-infinity}
\Rep\lambda_{n(k)} \le -k,\quad k\in\N.
\end{equation}
Let
\begin{equation*}
f:=\sum_{k=1}^\infty \frac{1}{k^2}e_{n(k)}\in X
\quad \text{and}\quad
g^*:=\sum_{k=1}^\infty \frac{1}{k^2}e_{n(k)}^*\in X^*,
\end{equation*}
the series strongly converging in $X$ and $X^*$, respectively, due to \eqref{vecseq} and \eqref{H-B}.
By \eqref{H-B} and \eqref{dist1},
\begin{equation}\label{dist3}
\langle e_{n(k)},g^*\rangle=\dfrac{1}{k^2}\langle e_{n(k)},e_{n(k)}^*\rangle=\dfrac{d_{n(k)}}{k^2}
\ge \dfrac{\varepsilon}{k^2},\quad k\in\N.
\end{equation}
By \eqref{ortho1},
\begin{equation}\label{subvectors2}
E_A(\Delta_{n(k)})f=\frac{1}{k^2}e_{n(k)},\quad k\in\N,\quad\text{and}\quad E_A(\cup_{k=1}^\infty\Delta_{n(k)})f=f.
\end{equation}
Considering the latter and \eqref{dist3},
\begin{equation}\label{vt2}
v(f,g^*,\Delta_{n(k)})\ge |\langle E_A(\Delta_{n(k)})f,g^* \rangle|
=\left\langle \frac{1}{k^2}e_{n(k)},g^* \right\rangle
\ge\dfrac{\varepsilon}{k^4},\quad k\in\N.
\end{equation}
Similarly to \eqref{similar}, for $s=1$ and $t=(2b)^{-1}$,
\begin{equation*}
\int_{\sigma(A)} T(|\lambda|)e^{(2b)^{-1}\Rep\lambda}\,dv(f,g^*,\lambda)
=\sum_{k=1}^\infty\int_{\Delta_{n(k)}}
T(|\lambda|)e^{(2b)^{-1}\Rep\lambda}\,dv(f,g^*,\lambda)=\infty.
\end{equation*}
Indeed, for $\lambda \in \Delta_{n(k)}$, $k\in\N$, by
\eqref{disks1}, \eqref{radii1}, and \eqref{-infinity},
\begin{multline*}
-bM(|\Imp\lambda|)<\Rep\lambda =\Rep\lambda_{n(k)}-(\Rep\lambda_{n(k)}-\Rep\lambda)\le
\Rep\lambda_{n(k)}+|\lambda_{n(k)}-\lambda|\\
\le \Rep\lambda_{n(k)}+\varepsilon_{n(k)}\le-k+1\le 0
\end{multline*}
and hence,
\begin{equation*}
\Rep\lambda\le -k+1\le 0 \quad \text{and}\quad |\lambda|\ge|\Imp\lambda|
\ge M^{-1}\left(b^{-1}[-\Rep\lambda]\right).
\end{equation*}
Using these estimates, for $k\in\N$, we have
\begin{multline*}
\int_{\Delta_{n(k)}}T(|\lambda|)e^{(2b)^{-1}\Rep\lambda}\,dv(f,g^*,\lambda)
\ge \int_{\Delta_{n(k)}}e^{M(|\lambda|)}e^{(2b)^{-1}\Rep\lambda}\,dv(f,g^*,\lambda)
\\
\shoveleft{
\ge \int_{\Delta_{n(k)}}
e^{M(M^{-1}(b^{-1}[-\Rep\lambda]))}e^{(2b)^{-1}\Rep\lambda}\,dv(f,g^*,\lambda)
}\\
\shoveleft{
=\int_{\Delta_{n(k)}}
e^{b^{-1}[-\Rep\lambda]}e^{(2b)^{-1}\Rep\lambda}\,dv(f,g^*,\lambda)
=\int_{\Delta_{n(k)}}
e^{(2b)^{-1}[-\Rep\lambda]}\,dv(f,g^*,\lambda)
}\\
\shoveleft{
\ge
e^{(2b)^{-1}(k-1)}v(f,g^*,\Delta_{n(k)})
\hfill
\text{by \eqref{vt2};}
}\\
\ge e^{(2b)^{-1}(k-1)}\dfrac{\varepsilon}{k^4}\to \infty\ \text{as $k\to\infty$}.
\end{multline*}
Thus, by {\cite[Proposition $3.1$]{Markin2002(1)}},
\begin{equation*}
e^{tA}f\bigr|_{t=(2b)^{-1}}\not\in D(T(|A|)).
\end{equation*}
Considering \eqref{CC}, the more so,
\begin{equation*}
e^{tA}f\bigr|_{t=(2b)^{-1}}\not\in \bigcap_{s>0}D(T(s|A|))
=C_{(m_n)}(A).
\end{equation*}
Hence, according to Proposition \ref{prop},
\begin{equation*}
e^{\cdot A}f\not\in C_{(m_n)}\left((0,\infty),X\right),
\end{equation*} 
which implies that the $C_0$-semigroup $\left\{e^{tA}\big|t\ge 0\right\}$ generated by $A$ is not a
$C_{(m_n)}$-semi\-group.

This concludes the analysis of all the possibilities and thus, the proof of the "only if" part by {\it contrapositive}.

By {\cite[Theorem 3.1]{Markin2015(1)}}, the function $T(\cdot)$ defined by \eqref{T} can be replaced with any {\it nonnegative}, {\it continuous}, and {\it increasing} on $[0,\infty)$ function $F(\cdot)$ satisfying \eqref{replace}.
\end{proof}

%%%%%%%%%%%%%%%%%%%%%%%%%%%%%%%%%%%%%%%%%%%%%%%%%%%%%%%%%%%%%
\section{Gevrey ultradifferentiable $C_0$-semigroups}

\begin{defn}
Let $\beta\ge 0$. We shall call a $C_0$-semigroup $\left\{S(t)\big|t\ge0\right\}$ in a complex Banach space $(X,\|\cdot\|)$ a Roumieu (Beurling) type Gevrey ultradifferentiable $C_0$-semigroup of order $\beta$, or a ${\mathcal E}^{\{\beta\}}$-semigroup (${\mathcal E}^{(\beta)}$-semigroup), if it is a $C_{\{[n!]^\beta\}}$-semigroup ($C_{([n!]^\beta)}$-semigroup, respectively) in accordance with Definition \ref{def1}.
\end{defn}
The sequence $m_n=[n!]^\beta$ with $\beta\ge 1$ satisfying the conditions {\bf (GR)} and {\bf (SBC)}
\cite{Markin2008} and the function $T(\cdot)$ being replaceable with $F(\lambda)=e^{\lambda^{1/\beta}}$,
 $\lambda\ge 0$, (see \cite{Markin2004(2)} for details), {\cite[Theorem $5.1$]{Markin2004(1)}}
 giving a characterization of the scalar type spectral generators of Roumieu type Gevrey ultradifferentiable
 $C_0$-semigroups of order $\beta\ge 1$ (in particular, for $\beta=1$, of \textit{analytic semigroups}
 \cite{Engel-Nagel,Hille-Phillips,Markin2002(2)}) immediately follows from Theorem \ref{thm}.

The sequence $m_n=[n!]^\beta$ with $\beta>1$
satisfying the conditions {\bf (SGR)} and {\bf (BC)} (see \cite{Markin2009} for details), in the same manner, a ready consequence of Theorem \ref{GBTCUS} is the following

\begin{cor}
Let $\beta>1$. Then a scalar type spectral operator $A$ in a complex Banach space $(X,\|\cdot\|)$ generates a ${\mathcal E}^{(\beta)}$-semigroup iff, for any $b>0$, there is an $a\in \R$ such that
\begin{equation*}
\Rep\lambda\le a-b|\Imp\lambda|^{1/\beta},\quad \lambda\in\sigma(A).
\end{equation*}
\end{cor}

Observe that, for $0\le \beta\le 1$, the sequence $m_n=[n!]^\beta$ fails to satisfy the condition {\bf (SGR)},
and, for $0\le \beta<1$, even {\bf (GR)}. If a scalar type spectral operator $A$ in a complex Banach space $(X,\|\cdot\|)$ generates a ${\mathcal E}^{\{\beta\}}$-semigroup with $0\le \beta<1$ or a ${\mathcal E}^{(\beta)}$-semigroup with $0\le \beta\le 1$,
due to inclusions \eqref{incl1} and \eqref{incl2},
\begin{equation*}
{\mathcal E}^{(\beta)}((0,\infty),X)
\subseteq {\mathcal E}^{\{\beta\}}((0,\infty),X)
\subseteq {\mathcal E}^{(1)}((0,\infty),X),
\end{equation*}
which implies that all the orbits $e^{\cdot A}f$, $f\in X$, are entire vector functions. Hence, being defined of the whole space $X$, $A\in L(X)$ by the \textit{Closed Graph Theorem} and generates a uniformly continuous semigroup (an \textit{entire semigroup of exponential type}).

%%%%%%%%%%%%%%%%%%%%%%%%%%%%%%%%%%%%%%%%%%%%%%%%%%%%%%%%%%%%%
\section{One more example}

The rapidly growing sequence $m_n:=e^{n^2}$ also satisfies the conditions {\bf (SGR)} and {\bf (BC)} and the function $M(\cdot)$ in this case can be replaced with
\begin{equation*}
L(\lambda):=\begin{cases}
0& \text{for $0\le\lambda<1$},\\
[\ln\lambda]^2& \text{for $\lambda\ge 1$}
\end{cases}
\end{equation*}
(see \cite{Markin2009} for details). Thus we have the following

\begin{cor}
A scalar type spectral operator $A$ in a complex Banach space $(X,\|\cdot\|)$ generates a $C_{(e^{n^2})}$-semigroup iff, for any $b>0$, there is an $a\in \R$ such that
\begin{equation*}
\Rep\lambda\le a-bL(|\Imp\lambda|),\quad \lambda\in\sigma(A).
\end{equation*}
\end{cor}

\noindent
(cf. {\cite[Theorem 4.1]{Markin2004(1)}}).
%%%%%%%%%%%%%%%%%%%%%%%%%%%%%%%%%%%%%%%%%%%%%%%%%%%%%%%%%%%%%
\section{Final remark}

Due to the scalar type spectrality of the operator $A$, Theorem \ref{GBTCUS} is void of restrictions on its
resolvent behavior, which appear to be inevitable for the results of this nature in the general case
(cf. \cite{Engel-Nagel,Hille-Phillips,Pazy}).

%%%%%%%%%%%%%%%%%%%%%%%%%%%%%%%%%%%%%%%%%%%%%%%%%%%%%%%%%%%%%

{\it{Acknowledgments}}. The author would like to express appreciation to his colleagues at the Department of Mathematics and the College of Science and Mathematics at California State University, Fresno for being a vibrant and stimulating scholarly community.

%%%%%%%%%%%%%%%%Bibliography%%%%%%%%%%%%%%%%%%%%%%%%%%%%%%%%%

%%%%%%%%%%%%%%%%%%%%%%%%%

\begin{thebibliography}{100000}



%%%%%%%%%%%%%%%%%%%%%%%%%%%%%%%%%%%%%%%%%%%%%%%%%%%%%%%%%%%%%
\bibitem{Ball}
{J. M. Ball},    %authors
\textit{Strongly continuous semigroups, weak solutions, and
the variation of constants formula},    %paper
{Proc. Amer. Math. Soc.}    %journal
\textbf{63} %vol
{(1977)},   %year
{no.~2},    %issue
{101--107}. %pages
%%%%%%%%%%%%%%%%%%%%%%%%%%%%%%%%%%%%%%%%%%%%%%%%%%%%%%%%%%%%%
\bibitem{Berkson1966}
{E. Berkson},   %authors
\textit{Semi-groups of scalar type operators and a theorem of Stone},   %paper
{Illinois J.~Math.} %journal
\textbf{10} %vol
{(1966)},   %year
{no.~2},    %issue
{345--352}.   %pages
%%%%%%%%%%%%%%%%%%%%%%%%%%%%%%%%%%%%%%%%%%%%%%%%%%%%%%%%%%%%%
\bibitem{Carleman}
{T. Carleman},  %authors
\textit{\'{E}dition Compl\`{e}te des Articles de Torsten Carleman}, %book
{Institut Math\'{e}ma\-tique Mittag-Leffler},   %bookinfo
{Djursholm},    %publ
{Su\`{e}de},        %publaddr
{1960}.     %year
%%%%%%%%%%%%%%%%%%%%%%%%%%%%%%%%%%%%%%%%%%%%%%%%%%%%%%%%%%%%%
\bibitem{Survey58}
{N. Dunford},   %authors
\textit{A survey of the theory of spectral operators}, %paper
{Bull. Amer. Math. Soc.}    %journal
\textbf{64} %vol
{(1958)},   %year
%{no.~4},   %issue
{217--274}. %pages
%%%%%%%%%%%%%%%%%%%%%%%%%%%%%%%%%%%%%%%%%%%%%%%%%%%%%%%%%%%%%
\bibitem{Dun-SchI}
{N. Dunford and J. T. Schwartz with the assistance of W. G. Bade and R. G. Bartle},    %authors
% with the assistance of W. G. Bade and R. G. Bartle
% },    %authors
\textit{Linear Operators. Part {\rm I}: General Theory},  %book
%{Pure and Applied Mathematics, vol. 7},    %bookinfo
{Interscience Publishers},  %publ
{New York},     %publaddr
{1958}.     %year
%%%%%%%%%%%%%%%%%%%%%%%%%%%%%%%%%%%%%%%%%%%%%%%%%%%%%%%%%%%%%
\bibitem{Dun-SchII}
{N. Dunford and J. T. Schwartz},  %authors
\textit{Linear Operators. Part {\rm{II}}: Spectral Theory. Self Adjoint Operators in Hilbert Space}, %book
%{Ibid.},   %bookinfo
{Interscience Publishers},  %publ
{New York},     %publaddr
{1963}.     %year
%%%%%%%%%%%%%%%%%%%%%%%%%%%%%%%%%%%%%%%%%%%%%%%%%%%%%%%%%%%%%
\bibitem{Dun-SchIII}
{N. Dunford and J. T. Schwartz},  %authors
\textit{Linear Operators. Part {\rm{III}}: Spectral Operators}, %book
%{Ibid.},   %bookinfo
{Interscience Publishers},  %publ
{New York},     %publaddr
{1971}.     %year
%%%%%%%%%%%%%%%%%%%%%%%%%%%%%%%%%%%%%%%%%%%%%%%%%%%%%%%%%%%%%
\bibitem{Engel-Nagel}
{K.-J. Engel and R. Nagel}, %authors
\textit{One-Parameter Semigroups
for Linear Evolution Equations}, %book
{Graduate Texts in Mathematics, vol. 194},  %bookinfo
{Springer-Verlag},  %publ
{New York},     %publaddr
{2000}.     %year
%%%%%%%%%%%%%%%%%%%%%%%%%%%%%%%%%%%%%%%%%%%%%%%%%%%%%%%%%%%%%
\bibitem{Gevrey}
{M. Gevrey},    %authors
\textit{Sur la nature analytique des solutions des  \'{e}quations aux d\'{e}riv\'{e}es partielles}, %paper
{Ann. Ec. Norm. Sup. Paris} %journal
\textbf{35} %vol
{(1918)},   %year
%{no.~4},   %issue
{129--196}. %pages
%%%%%%%%%%%%%%%%%%%%%%%%%%%%%%%%%%%%%%%%%%%%%%%%%%%%%%%%%%%%%
\bibitem{Goodman}
{R. Goodman},   %authors
\textit{Analytic and entire vectors for representations of Lie groups}, %paper
{Trans. Amer. Math. Soc.}   %journal
\textbf{143}    %vol
{(1969)},   %year
%{no.~4},   %issue
{55-76}.    %pages
%%%%%%%%%%%%%%%%%%%%%%%%%%%%%%%%%%%%%%%%%%%%%%%%%%%%%%%%%%%%%
\bibitem{GorV83}
{V. I. Gorbachuk},   %authors
\textit{Spaces of infinitely differentiable vectors of a nonnegative self-adjoint operator},    %paper
{Ukrain. Mat. Zh.}    %journal
\textbf{35} %vol
{(1983)},   %year
{no.~5},   %issue
{617--621. (Russian); English transl.} %pages
{Ukrainian Math.~J.}    %journal
\textbf{35} %vol
{(1983)},   %year
{no.~5},   %issue
{531--534}. %pages
%%%%%%%%%%%%%%%%%%%%%%%%%%%%%%%%%%%%%%%%%%%%%%%%%%%%%%%%%%%%%
\bibitem{book}
{V. I. Gorbachuk and M. L. Gorbachuk},    %authors
\textit{Boundary Value Problems for Operator Differential Equations},   %book
%{Mathematics and Its Applications (Soviet Series), vol.~48},    %bookinfo
{Kluwer Academic Publishers}, %publ
{Dordrecht---Boston---London},        %publaddr
{1991. (Russian edition: Naukova Dumka, Kiev, 1984)}     %year
%%%%%%%%%%%%%%%%%%%%%%%%%%%%%%%%%%%%%%%%%%%%%%%%%%%%%%%%%%%%
\bibitem{Gor-Knyaz}
{V. I. Gorbachuk and A. V. Knyazyuk}, %authors
\textit{Boundary values of solutions of operator-differential equations}, %paper
{Russian Math. Surveys}   %journal
\textbf{44} %vol
{(1989)},   %year
{no.~3},   %issue
{67--111}.  %pages
%%%%%%%%%%%%%%%%%%%%%%%%%%%%%%%%%%%%%%%%%%%%%%%%%%%%%%%%%%%%%
\bibitem{Hille-Phillips}
{E. Hille and R. S. Phillips},   %authors
\textit{Functional Analysis and Semi-groups},   %book
{American Mathematical Society Colloquium Publications, vol.~31},   %bookinfo
{Amer. Math. Soc.}, %publ
{Providence, RI},       %publaddr
{1957}.     %year
%%%%%%%%%%%%%%%%%%%%%%%%%%%%%%%%%%%%%%%%%%%%%%%%%%%%%%%%%%%%%
\bibitem{Komatsu}
{H. Komatsu} %authors
\textit{Ultradistributions. {\rm I}.
Structure theorems and characterization},   %paper
{J.~Fac. Sci. Univ. Tokyo}  %journal
\textbf{20} %vol
{(1973)},   %year
%{no.~4},   %issue
{25--105}.  %pages
%%%%%%%%%%%%%%%%%%%%%%%%%%%%%%%%%%%%%%%%%%%%%%%%%%%%%%%%%%%%%
\bibitem{Mandel}
{S. Mandelbrojt},   %authors
\textit{Series de Fourier et Classes Quasi-Analytiques de Fonctions},   %book
%{bookinfo},    %bookinfo
{Gauthier-Villars}, %publ
{Paris},        %publaddr
{1935}.     %year
%%%%%%%%%%%%%%%%%%%%%%%%%%%%%%%%%%%%%%%%%%%%%%%%%%%%%%%%%%%%%
\bibitem{Markin2002(1)}
{M. V. Markin},  %authors
\textit{On an abstract evolution equation with a spectral operator of scalar type}, %paper
{Int. J.~Math. Math. Sci.}  %journal
\textbf{32} %vol
{(2002)},   %year
{no.~9},    %issue
{555--563}. %pages
%%%%%%%%%%%%%%%%%%%%%%%%%%%%%%%%%%%%%%%%%%%%%%%%%%%%%%%%%%%%%
\bibitem{Markin2002(2)}
{M. V. Markin},  %authors
\textit{A note on the spectral operators of scalar type and semigroups of bounded linear operators},    %paper
{Ibid.} %journal
\textbf{32} %vol
{(2002)},   %year
{no.~10},   %issue
{635--640}. %pages
%%%%%%%%%%%%%%%%%%%%%%%%%%%%%%%%%%%%%%%%%%%%%%%%%%%%%%%%%%%%%
\bibitem{Markin2004(1)}
{M. V. Markin},  %authors
\textit{On scalar type spectral operators, infinite differentiable and Gevrey ultradifferentiable $C_0$-semigroups},    %paper
{Ibid.} %journal
\textbf{2004}   %vol
{(2004)},   %year
{no.~45},   %issue
{2401--2422}.   %pages
%%%%%%%%%%%%%%%%%%%%%%%%%%%%%%%%%%%%%%%%%%%%%%%%%%%%%%%%%%%%%
\bibitem{Markin2004(2)}
{M. V. Markin},  %authors
\textit{On the Carleman classes of vectors of a scalar type spectral operator}, %paper
{Ibid.} %journal
\textbf{2004}   %vol
{(2004)},   %year
{no.~60},   %issue
{3219--3235}.   %pages
%%%%%%%%%%%%%%%%%%%%%%%%%%%%%%%%%%%%%%%%%%%%%%%%%%%%%%%%%%%%%
\bibitem{Markin2008}
{M. V. Markin},  %authors
\textit{On scalar type spectral operators and Carleman ultradifferentiable $C_0$-se\-migroups}, %paper
{Ukrain. Mat. Zh.}    %journal
\textbf{60} %vol
{(2008)},   %year
{no.~9},    %issue
{1215--1233; English transl.}   %pages
{Ukrainian Math.~J.}    %journal
\textbf{60} %vol
{(2008)},   %year
{no.~9},    %issue
{1418--1436}.   %pages
%%%%%%%%%%%%%%%%%%%%%%%%%%%%%%%%%%%%%%%%%%%%%%%%%%%%%%%%%%%%%
\bibitem{Markin2009}
{M. V. Markin},  %authors
\textit{On the Carleman ultradifferentiability of weak solutions of an abstract evolution equation},    %paper
{Modern Analysis and Applications}, %book
{Oper. Theory Adv. Appl.},  %bookinfo
{vol.~191}, %vol
%{2009},        %year
{Birkh\"a\-user Verlag},    %publ
{Basel},        %publaddr
{2009},
{pp.~407--443}. %pages
%%%%%%%%%%%%%%%%%%%%%%%%%%%%%%%%%%%%%%%%%%%%%%%%%%%%%%%%%%%%%
\bibitem{Markin2015(1)}
{M. V. Markin},  %authors
\textit{On the Carleman ultradifferentiable vectors of a scalar type spectral operator},    %paper
{Methods Funct. Anal. Topology}   %journal
\textbf{21}    %vol
{(2015)},  %year
{no.~4},   %issue
{361--369}.  %pages
%%%%%%%%%%%%%%%%%%%%%%%%%%%%%%%%%%%%%%%%%%%%%%%%%%%%%%%%%%%%%
\bibitem{Nelson}
{E. Nelson},    %authors
\textit{Analytic vectors},  %paper
{Ann. Math. (2)} %journal
\textbf{70} %vol
{(1959)},   %year
{no.~3},    %issue
{572--615}. %pages


%E. Nelson, Analytic vectors, Ann. of Math. 70 (1959), no. 2,
%572-615.

%%%%%%%%%%%%%%%%%%%%%%%%%%%%%%%%%%%%%%%%%%%%%%%%%%%%%%%%%%%%%
\bibitem{Panchapagesan1969}
{T. V. Panchapagesan},   %authors
\textit{Semi-groups of scalar type operators in Banach
spaces},    %paper
{Pacific J.~Math.}  %journal
\textbf{30} %vol
{(1969)},   %year
{no.~2},    %issue
{489--517}. %pages
%%%%%%%%%%%%%%%%%%%%%%%%%%%%%%%%%%%%%%%%%%%%%%%%%%%%%%%%%%%%%
\bibitem{Pazy1968}
{A. Pazy},  %authors
\textit{On the differentiability and compactness of semigroups of linear operators},    %paper
{J.~Math. Mech.}    %journal
\textbf{17} %vol
{(1968)},   %year
%{no.~2},   %issue
{1131--1141}. %pages
%%%%%%%%%%%%%%%%%%%%%%%%%%%%%%%%%%%%%%%%%%%%%%%%%%%%%%%%%%%%%
\bibitem{Pazy}
{A. Pazy},  %authors
\textit{Semigroups of Linear Operators and Applications to Partial Differential Equations}, %book
{Appl. Math. Sci., vol.~44},    %bookinfo
{Springer-Verlag},  %publ
{New York},     %publaddr
{1983}.     %year
%%%%%%%%%%%%%%%%%%%%%%%%%%%%%%%%%%%%%%%%%%%%%%%%%%%%%%%%%%%%%
\bibitem{Plesner}
{A. I. Plesner}, %authors
\textit{Spectral Theory of Linear Operators},   %book
%{},    %bookinfo
{Nauka},    %publ
{Moscow},       %publaddr
{1965}.      %year
{(Russian)}
%%%%%%%%%%%%%%%%%%%%%%%%%%%%%%%%%%%%%%%%%%%%%%%%%%%%%%%%%%%%%
\bibitem{Radyno1983(1)}
{Ya. V. Radyno}, %authors
\textit{The space of vectors of exponential type},  %paper
{Dokl. Akad. Nauk BSSR} %journal
\textbf{27} %vol
{(1983)},   %year
{no.~9},    %issue
{791--793}.    %pages
{(Russian with English summary)}
%%%%%%%%%%%%%%%%%%%%%%%%%%%%%%%%%%%%%%%%%%%%%%%%%%%%%%%%%%%%%
\bibitem{Wermer}
{J. Wermer},    %authors
\textit{Commuting spectral measures on Hilbert space},  %paper
{Pacific J.~Math.}  %journal
\textbf{4}  %vol
{(1954)},   %year
{no.~3},    %issue
{355--361}. %pages
%%%%%%%%%%%%%%%%%%%%%%%%%%%%%%%%%%%%%%%%%%%%%%%%%%%%%%%%%%%%%
\bibitem{Yosida1958}
{K. Yosida},    %authors
\textit{On the differentiability of semi-groups of linear operators},   %paper
{Proc. Japan Acad.} %journal
\textbf{34} %vol
{(1958)},   %year
{no.~6},    %issue
{337--340}. %pages
%%%%%%%%%%%%%%%%%%%%%%%%%%%%%%%%%%%%%%%%%%%%%%%%%%%%%%%%%%%%%
\bibitem{Yosida}
{K. Yosida},    %authors
\textit{Functional Analysis},   %book
{6th ed.},  %bookinfo
{Springer-Verlag},  %publ
{New York},     %publaddr
{1980}.     %year
%%%%%%%%%%%%%%%%%%%%%%%%%%%%%%%%%%%%%%%%%%%%%%%%%%%%%%%%%%%%%
\end{thebibliography}
\end{document}